\documentclass[11pt]{article}
\usepackage[a4paper]{anysize}\marginsize{3.1cm}{3.1cm}{0.9cm}{3cm}
   \pdfpagewidth=\paperwidth \pdfpageheight=\paperheight
\usepackage{amssymb,amsmath}
\usepackage{graphicx}
\renewcommand\ge{\geqslant}  
\renewcommand\le{\leqslant}  

\newcommand\be{\begin{eqnarray*}}
\newcommand\ee{\end{eqnarray*}}

\newtheorem{introtheorem}{Theorem}
\newtheorem{theorem}{Theorem}[section]
\newtheorem{conjecture}{Conjecture}[section]

\newtheorem{lemma}[theorem]{Lemma}
\newtheorem{definition}[theorem]{Definition}

\newtheorem{example}[theorem]{Example}

\newenvironment{proof}[1][Proof]{\trivlist\item[\hskip\labelsep{\textit{#1.}}]}{\hspace*{\fill}$\Box$\endtrivlist}

\newcommand\PP{\mathbb P}
\newcommand\C{\mathbb C}
\newcommand\Z{\mathbb Z}
\newcommand\N{\mathbb N}
\newcommand\R{\mathbb R}
\newcommand{\cC}{\mathcal C}

\newcommand\numequiv{\equiv_{\rm num}}

\newcommand\vol{{\rm vol}}
\newcommand\dvol{d_{\rm vol}}

\newcommand\bprim{b_{\rm prim}}
\newcommand\NS{{\rm NS}}

\newcommand{\equ}{\ensuremath{\,=\,}}

\newcommand{\deq}{\ensuremath{\stackrel{\textrm{def}}{=}}}

\DeclareMathOperator{\Nef}{Nef}
\DeclareMathOperator{\Eff}{Eff}
\DeclareMathOperator{\vdim}{vdim}

\begin{document}

\title{Bounded volume denominators and bounded negativity}
\author{\normalsize Th. Bauer, B. Harbourne, A. K\"uronya, M. Nickel}
\maketitle
\thispagestyle{empty}


\begin{abstract}
   In this paper we study the question of whether
   on smooth projective surfaces the denominators
   in the volumes of big line bundles are bounded.
   In particular we investigate how this condition
   is related to bounded negativity (i.e., the boundedness
   of self-intersections of irreducible curves).
   Our first result shows that
   boundedness of volume denominators is
   equivalent to \emph{primitive bounded negativity}, which
   in turn is implied by bounded negativity.
   We connect this result to the study of semi-effective
   orders of divisors:
   Our second result shows that negative
   classes exist that become effective only after
   taking an arbitrarily large
   multiple.
\end{abstract}


\section*{Introduction}

Boundedness conditions of various kinds have generated vivid
   interest in algebraic geometry. One such that has received recent attention
   is the Bounded Negativity Problem: for which smooth, projective surfaces $X$
   is there no lower bound for $C^2$
   for the set of reduced (or, if you prefer, reduced and irreducible)
   curves $C$ on $X$? The only examples of such an $X$ currently known
   are in positive characteristics, and even then they are very special.

Surfaces for which a bound exists are said to have \emph{bounded negativity}.
  The \emph{Bounded Negativity Conjecture} asserts that
   over the complex numbers
   every smooth projective surface should have bounded negativity
   (see \cite{Bauer-et-al:Negative-curves, 11authors, MT} for
   the history of the conjecture and for results towards the
   conjecture, and \cite{Cilib-et-al} for related boundedness questions).

Another boundedness property that has gotten attention is
the boundedness of the denominators that appear
in Zariski decompositions of pseudo-effective divisors on $X$.
This was shown in \cite{BPS} to be equivalent to bounded negativity.

   We now want to ask what may initially seem to be a rather different problem, which we refer to
   as the Bounded Negative Semi-effective Problem. To state the problem, we will
   say that a divisor $F$ is \emph{semi-effective}
   if $kF$ is numerically equivalent to an effective divisor for some $k\geq1$.
   The least such $k$ is called the \emph{semi-effective order} of $F$.
   The Bounded Negative Semi-effective Problem is now:
   what values $k$ arise as the semi-effective order of a divisor $F$ on some surface $X$
   such that $kF$ is numerically equivalent to a prime divisor of negative self-intersection?
   In particular, is there an upper bound on such semi-effective orders $k$ for a given $X$?
   If for some $X$ there is no bound, then, since we always would have $F^2\leq-1$, $X$ would not have bounded negativity.

   Of course, since we're working up to numerical equivalence and
   looking for large semi-effective orders, it makes sense to pick a divisor $F$
   whose class is primitive (so not $kG$ for some class $G$ and some integer $k\geq 2$).
   I.e., it does no harm to assume $F$ is primitive:
   if $F$ is not primitive to start with, then up to numerical equivalence $F=kG$ for some primitive $G$,
   so we might as well replace $F$ by $G$. Thus the Bounded Negative Semi-effective Problem
   is equivalent to the following:
   for which surfaces $X$ is there a bound on the semi-effective orders of
   primitive numerical equivalence classes whose ray contains the class of a prime divisor of
   negative self-intersection?

  Again, if there is a surface $X$ with numerically primitive classes $F_1, F_2,\dots,$ with $F_i^2<0$ and
  unbounded semi-effective orders $k_1,k_2,\dots,$ and such that $k_iF_i$ is numerically a prime divisor,
  then $X$ has unbounded negativity. But perhaps the set of values $F_i^2$ is bounded.
  To explore this we introduce the notion of \emph{bounded primitive negativity},
  and, motivated by the results of \cite{BPS}, we prove the following theorem, which relates it to
  a further natural boundedness condition: we ask
   whether the denominators that appear in the volumes
   $\vol(L)$ of big line bundles $L$ on $X$ are bounded, which we call \emph{boundedness of volume denominators}.

\begin{introtheorem}
   For a smooth projective surface $X$, the following conditions
   are equivalent:
   \begin{itemize}
   \item[\rm(i)]
      $X$ has bounded volume denominators, i.e., there exists an
      integer $\dvol(X)$ such that for every (integral) big line
      bundle $L$, the volume
      \be
         \vol_X(L)=\limsup_k\frac{h^0(X,kL)}{k^2/2}
      \ee
      is a rational number with denominator at most $\dvol(X)$.
   \item[\rm(ii)]
      $X$ has bounded primitive negativity, i.e., there exists an
      integer $\bprim(X)$ such that for every primitive
      class $F\in\NS(X)$, whose ray $\R^+\cdot F$ contains a reduced
      irreducible curve, one has
      \be
         F^2 \ge -\bprim(X)
      \ee
   \end{itemize}
\end{introtheorem}

In view of \cite{BPS} and Theorem 1, we have the following
   diagram of
   implications.
   $$
      \newcommand\pbox[1]{\fbox{\parbox{10em}{\strut\centering #1\strut}}}
      \newcommand\jo[1]{\noalign{\vskip#1\jot}}
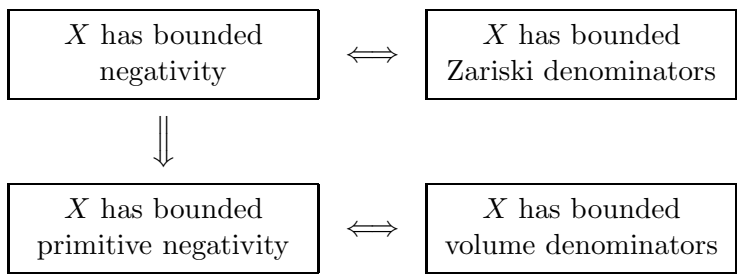
      \begin{array}{ccc}
         \pbox{$X$ has bounded negativity} & \Longleftrightarrow & \pbox{$X$ has bounded Zariski denominators} \\ \jo2
         \Big\Downarrow & & \\ \jo2
         \pbox{$X$ has bounded primitive negativity} & \Longleftrightarrow & \pbox{$X$ has bounded volume denominators}
      \end{array}
   $$
   It is not clear whether the converse of the downward vertical implication
   is true. We do not even know if a prime divisor of negative self-intersection
   is ever numerically equivalent to $kF$ for a primitive class $F$ with $k>2$
   (examples with $k=2$ have long been known).
   In fact, dropping the negativity assumption still leaves an open question:
   is there a surface $X$ and an infinite sequence of primitive classes $F_i$ on $X$
   with semi-effective orders $k_1<k_2<\dots$?
   The article \cite{Cilib-et-al} shows the answer is yes, assuming the SHGH Conjecture
   (stated below).
   Although the divisors given in \cite{Cilib-et-al} are nef, in Example \ref{PellExample} we show they can be modified to
   give primitive classes $F_i$ with $F_i^2<0$ with the same semi-effective orders $k_1<k_2<\dots$
   (all under the assumption of the SHGH Conjecture).

   Thus it is expected that there are surfaces having primitive classes with unbounded semi-effective orders,
   even if one requires them to have negative self-intersection.
   However, no surfaces are known having an infinite sequence of primitive classes with both unbounded semi-effective orders
   and having effective multiples with negative definite intersection matrix.
   Our best result along these lines is the following theorem. It shows for each $N$ as large as you like that there is a surface
   $X$ depending on $N$ and an integer $k\geq N$ for which there is a divisor $F$ on $X$
   of semi-effective order $k$ such that $kF$ is
   numerically equivalent to an effective divisor $D$ having negative
   definite intersection matrix.
In these examples, $D$ is not irreducible
(and, in characteristic 0, not even reduced).

\begin{introtheorem}\label{introthm:divisors}
   Over any algebraically closed field $K$, there exists a sequence of
   smooth projective surfaces $X_n$ and primitive classes $F_n$ of
   semi-effective orders $k_n$ with
   $$\lim_{n\to\infty} k_n=\infty$$
   such that $k_nF_n$ is numerically equivalent to
   an effective divisor $D_n\subset X_n$ having a negative definite intersection matrix.
\end{introtheorem}

\paragraph{Acknowledgements.} Part of this work was done while the first, second, and third authors enjoyed the hospitality of the Mathematisches Forschungsinstitut in Oberwolfach as
participants of the Mini-Workshops  `Arrangements of Subvarieties, and their Applications in Algebraic Geometry'  and
`PBW Structures in Representation Theory'. We would like to use this opportunity to thank the organizers of the latter event, Evgeny Feigin, Ghislain Fourier, and Martina Lanini, for the invitation, and the MFO for the traditionally
inspiring atmosphere and the excellent working environment.

The third author was partially supported by the LOEWE grant `Uniformized Structures in Algebra and Geometry'.


\section{Bounded primitivity}

We now formally state the concept of bounded primitivity for divisor classes on projective varieties.
In case that $\cC = N^1(X)_\Z$, this is the same as $X$ having bounded semi-effective orders.

\begin{definition}[Bounded primitivity]
Let $X$ be a projective variety, $\cC\subseteq N^1(X)_\Z$ an arbitrary subset. We say that $\cC$ \emph{satisfies bounded primitivity (BP)} if there exists $m_0=m_0(\cC)\in\N$ such that for every $\alpha\in\cC$ there exists $1\leq i\leq m_0$ such that $i\alpha$ is numerically equivalent to an effective divisor.
\end{definition}

Let us look at a concrete class of examples: very general blow-ups of the projective plane. As we will see, although little is known about bounded primitivity of negative curves or bounded negativity on blow-ups of $\PP^2$ (except when they are  weak Fano), the Segre--Harbourne--Gimigliano--Hirschowitz (SHGH) Conjecture (\cite{S,Ha,G,Hi} yields a very strong conjectural answer.

Let $f\colon X_r\to \PP^2$ be the blow-up of $\PP^2$ at $r$ \emph{very general} points with exceptional divisors $E_1,\cdots,E_r$, write $H$ for the pullback of the hyperplane class on $X_r$. Following tradition, if
\[
D \equ dH - \sum_{i=1}^{r} m_iE_i \equ : (d;m_1,\dots,m_r),
\]
then the \emph{virtual dimension} of $D$ is defined as
\[
\vdim(D) \deq \frac{(D\cdot (D-K_{X_r}))}{2} \equ \frac{d(d+3)}{2} - \sum_{i=1}^{r} \frac{m_i(m_i+1)}{2}\ .
\]
For simplicity we will assume that $m_1\geq m_2\geq \ldots\geq m_r>0$.

\begin{conjecture}[SHGH Conjecture]
With notation as above, assume that $r\geq 3$ and  $D=(d;m_1,\ldots,m_r)$ with $d>m_1+m_2+m_3$. Then
\[
\dim |D| \equ \max \{ -1, \vdim(D)\}\ .	
\]
\end{conjecture}

The SHGH Conjecture has been verified for $r\leq 9$ points. It also is known that it implies the following conjecture \cite{Ha,CM}.

\begin{conjecture}[Negative curves on very general blow-ups of the plane]
Let $X_r$ be a blow-up of $\PP^2$ at $r\geq 10$ very general points, and let $C\subset X$ be a reduced
irreducible curve with negative self-intersection. Then $(C^2)=-1$. In particular, BNC holds on $X_r$ with the bound $b(X)=1$, and every negative curve on $X_r$ is primitive.
\end{conjecture}

Interestingly enough, there is an unboundedness result for nef classes  in the context of the  SHGH Conjecture.
We recall the construction, which is an elegant application of number theory, from \cite{Cilib-et-al},
and we show that it gives examples of primitive classes of negative self-intersection
with unbounded semi-effective orders.

\begin{example}[Neither nef nor negative classes have bounded primitivity]\label{PellExample}\rm
Let $X'\deq X_{10}$. We will consider homogeneous divisor classes $D=(d;m^{10})$. The starting point of the construction is the Pell equation
\[
x^2 - 10y^2 \equ -1\ ,
\]	
where $x=2d+3$ and $y=2m+1$. For a natural number $k$, write $C_k\deq p_k/q_k$ for the $k$\textsuperscript{th} convergent of the simple continued fraction of $\sqrt{10}=[3;\overline{6}]$,
with $\gcd(p_k,q_k)=1$. One has $(p_0,q_0)=(3,1)$, and in general
\[
(3+\sqrt{10})^{k+1} \equ p_k + q_k\sqrt{10}\ \ \text{for all $k\in\N$.}
\]	
With this said, for all natural numbers $k$, set
\[
d_k \deq \frac{p_{2k}-3}{2}\ ,\ m_k \deq \frac{q_{2k}-1}{2}\ ,\ \text{ and }\ D_k \deq (d_k;m_k^{10})\ .
\]
The divisors $D_k$ are called \emph{Pell divisors} in \cite[Section 3]{Cilib-et-al}.
By the Divisibility Lemma \cite[Lemma 2]{Cilib-et-al}, for every $k\geq 0$ there exists a divisor $F_k$ on $X'$ such that $D_k=c_kF_k$ and
$c_k=p_{k-1}$,$F_k=(p_k;q_k^{10})$ for $k$ odd, while $c_k=q_{k-1}$,$F_k=(10q_k;p_k^{10})$ whenever $k$ is even. According to the proof of \cite[Proposition 4]{Cilib-et-al}, which assumes the SHGH Conjecture,
the primitivity of the nef divisor class is precisely $c_k$, but $c_k$ (being part of the solution to a Pell equation) can be an arbitrarily large positive number as $k\to \infty$.

Now let $X$ be the blow up of $X'$ at one point, away from the other points blown up.
We can regard the divisors on $X'$ as being a subgroup of those on $X$ in the obvious way.
Let $E$ be the exceptional divisor of this last blow up. Let $s_k>\sqrt{D_k^2}$ be an integer, so $(D_k+s_kE)^2 < 0$.
Then $D_k+s_kE$ is primitive and
has semi-effective order equal to that of $D_k$. Thus the classes $D_k+s_kE$ have unbounded semi-effective orders $c_k$
but $c_k(D_k+s_kE)$ is effective with negative self-intersection.
\end{example}

In contrast to the Bounded Negativity Conjecture, there are in general signs that positive divisors on surfaces tend not to obey strong boundedness conditions. A classic example of Koll\'ar, recalled from \cite[Example 1.5.7]{Laz1}, points out the lack of uniform very ampleness.

\begin{example}{\rm
We construct a smooth projective surface $Y$ and a sequence of ample line bundles $D_m$ on $X$ such
that $m\cdot D_m$ is \emph{not} very ample.
We start with an elliptic curve $C$ and consider the self-product $X\deq C\times C$. Then
a basis of the N\'eron--Severi space $N^1(X)_\R$ is given by the classes $f_1,f_2$ of the fibers of the natural projections,
and the diagonal $\delta$.  The nef and effective cones of $X$ are known (see \cite[Lemma 1.5.4]{Laz1} for instance), and we have
\[
\Nef(X) \equ \overline{\Eff}(X) \equ \{\alpha\in N^1(X)_\R\mid (\alpha^2)\geq0\, ,\, (\alpha\cdot (f_1+f_2))\geq 0 \}\ .
\]
For every integer $m\geq 2$, define the integral divisor
\[
A_m \deq m\cdot F_1 + (m^2-m_1)\cdot F_2 - (m-1)\cdot \Delta.
\]
Since $(A_m^2)=2>0$ and $(A_m\cdot(F_1+F_2))>0$, $A_m$ is ample for all $m\geq 2$. 	

Next, set $R\deq F_1+F_2$, let $B\in |2R|$ be a smooth divisor, and take $Y$ to be
the double cover $f\colon Y\to X$ branched along the divisor $B$. Take $D_m\deq f^*A_m$, then $D_m$ is ample,
but $m\cdot D_m$ is not very ample since it does not separate points of $Y$ sitting in fibers of $f$.
}
\end{example}


\section{Bounded volume denominators and bounded primitive negativity}

\begin{proof}[Proof of Theorem 1]
   The proof is based on the fact that for big $L$ with Zariski
   decomposition $L=P+N$, the volume is computed by
   \be
      \vol_X(L)=P^2
      \,.
   \ee
   We start by proving the implication (i) $\Rightarrow$ (ii):
   We assume that $X$ has unbounded
   primitive negativity. We will show the contrapositive, that $X$ has
   unbounded volume denominators.

   By assumption then, there
   exists a sequence
   $C_n$ of irreducible curves
   and primitive classes $F_n$
   on $X$ with
   $C_n\numequiv k_nF_n$ and
   $F_n^2\to-\infty$.
   Let $A$ be any ample line bundle on~$X$.
   For every integer $n>0$ we choose an integer $m_n$ bigger than the
   rational number
   \be
      \alpha_n := -\frac{AC_n}{C_n^2}
      \,.
   \ee
   Then the
   big (integral) divisor
   $L_n:=A+m_nC_n$ has
   Zariski decomposition
   \be
      L_n=(A+\alpha_nC_n)+(m_n-\alpha_n)C_n
   \ee
   and therefore its
   volume is given by
   \be
      \vol_X(L_n)= (A+\alpha_nC_n)^2
      =\Big(A^2-\frac{(AC_n)^2}{C_n^2}\Big)
      =\Big(A^2-\frac{(AF_n)^2}{F_n^2}\Big)
      \,.
   \ee
   The
   issue thus becomes to show that the denominators of
   the fractions $\frac{(AF_n)^2}{F_n^2}$
   are unbounded.
   We assert:
   \begin{itemize}
   \item[(*)]
      For every $n$, there exists an
      ample line bundle $A_n$ such that
      \be
         \gcd(F_n^2, A_nF_n) \mbox{ divides } \Delta
         \,,
      \ee
   \end{itemize}
   where $\Delta$ is the discriminant of $\NS(X)$, i.e.,
   the determinant of the intersection matrix of a basis of $\NS(X)$.
   Granting (*), the assertion follows, since
   the denominator in $\frac{(A_nF_n)^2}{F_n^2}$ is
   \be
      \frac{-F_n^2}{\gcd(F_n^2,(A_nF_n)^2)} \ge
      \frac{-F_n^2}{|\Delta|^2}\to\infty
      \,.
   \ee
   Finally, assertion (*) follows from
   Lemma~\ref{lemma:dicriminant} below
   plus an argument as in \cite[Lemma~2.4]{BPS}.

   We now prove the implication
   (ii) $\Rightarrow$ (i):
   Assume that $X$ has bounded primitive negativity,
   with bound $\bprim(X)$.
   Let $L$ be any big divisor on $X$,
   and consider its Zariski decomposition
   \be
      L=P+N=P+\sum_i a_iC_i
   \ee
   where the $C_i$ are irreducible curves on $X$.
   Writing $C_i=k_iF_i$ with integers $k_i$ and
   primitive classes $F_i$, we have
   \be
      L=P+\sum_i a_ik_iF_i
   \ee
   We will now generalize the argument from
   \cite[Sect.~1]{BPS} in order to see that the
   coefficients $a_ik_i$ have bounded denominators.
   (Note that this does not imply that the
   Zariski denominators, i.e., the denominators of the $a_i$,
   are necessarily bounded as well, because of the possible effect
   of multiplying $a_i$ by $k_i$.)

   The values of the products $a_ik_i$ are given as the unique solution
   of the system of linear equations
   \be
      L\cdot F_j = (P+\sum_i a_ik_i F_i)\cdot F_j = \sum_i a_ik_i F_iF_j
   \ee
   and therefore can be written as rational numbers with denominator
   $\det S$, where $S$ is the intersection matrix of
   the classes $F_i$. So we have shown so far that the
   denominators of the numbers $a_ik_i$ are bounded by
   $\det S$. Arguing now as in the proof of \cite[Theorem~2.2]{BPS},
   one finds that
   \be
      |\det S|\le \bprim(X)^{\rho(X)-1}
   \ee
   where $\rho(X)$ is the Picard number of $X$.
   Finally, since
   \be
      \vol(L)= P^2= (L-N)^2
      \,,
   \ee
   we see that the volume denominators are bounded by
   \be
      \dvol(X)=\bprim(X)^{2\rho(X)-2}
   \ee
\end{proof}

\begin{lemma}\label{lemma:dicriminant}
   Let $X$ be a smooth projective surface, and $F\in\NS(X)$ a
   primitive class. If an integer $t$ divides the intersection
   number $AF$ for all
   ample classes $A\in\NS(X)$, then $t$ divides the
   discriminant of $\NS(X)$.
\end{lemma}

   Note that this is
   a more general version of \cite[Lemma~2.5]{BPS} -- the
   essential novelty is that the assumption of bounded Zariski
   denominators (which is not satisfied in our situation) can be
   replaced by the assumption that $F$ be primitive.

\begin{proof}
   If $t$ divides $AF$ for all
   ample line classes $A\in\NS(X)$, then in fact it divides
   $DF$ for all $D\in\NS(X)$. Thinking now of
   $\NS(X)/\mbox{Torsion}$ as $\Z^{\rho(X)}$, the assumption is
   that $t$ divides
   the number
   \be
      f^t S d \qquad\mbox{for all } d\in\Z^{\rho(X)}
      \,,
   \ee
   where $f\in\Z^{\rho(X)}$ is the coordinate vector of $F$, and $S$ is
   the matrix of the intersection form on $\NS(X)$.
   Using the adjugate matrix $S_{\rm adj}$, we get that $t$ divides
   \be
      f^t SS_{\rm adj}=f^t \det S
      \,,
   \ee
   and, as $f$ is primitive, this implies that $t$ divides
   $\det S$.
\end{proof}


\section{Primitivities of negative definite divisors}

   It has been known for a long time (see Example \ref{example:Sextic})
   that imprimitive prime divisors of negative self-intersection
   exist, but examples are rare. We do not know of an example with primitivity more than 2.

\begin{example}\label{example:Sextic}\rm
Here we mention the well known example of an irreducible plane sextic with 10 double points;
the double points can be taken to be nodes with real coordinates
(see \cite[p. 26, 51]{Hu}). Let $X$ be the surface obtained by blowing up the 10 double points.
Let $H$ be the class of a line and $E_i$ the exceptional divisor for the blow up of the $i$th point.
Then $M=3H-E_1-\dots-E_{10}$ is primitive but not effective, while $2M$ is the
proper transform of the sextic, hence irreducible. Thus $F$ has semi-effective order 2.

We note that $F$ is a fiber in a non-Jacobian elliptic fibration coming from a pencil of sextics,
the other generator being the twice the cubic through some choice of 9 of the 10 points blown up.
More generally we get additional examples by taking non-Jacobian rational elliptic surfaces $X'$
with a multiple fiber of multiplicity $m$ when there is a non-multiple fiber having a singular point
also of multiplicity $m$. Let $X$ be the blow up of $X'$ at the singular point with $E$ being the
exceptional curve.
Let $F$ be the multiple fiber (taken without multiplicity), then let $M=F-E$.
Then $m$ is the semi-effective order of $M$
and $mM$ is effective, contained in a fiber of the pencil $|mF|$.
It is possible for $m$ to be bigger than 2 (on a rational surface with a fiber of type $E_8$,
$m$ can be as big as 11), but thanks to the Kodaira classification of fibers of
elliptic fibrations, $m$ can't be bigger than 2 when $mF$ is a prime divisor.
See \cite{HL, HM} for background on non-Jacobian rational elliptic surfaces.
\end{example}

\begin{example}\label{example:LowPrimitivity2}\rm
The morphism $\PP^1\to \PP^2$ given by three general
homogeneous forms of degree $d$ in $K[\PP^2]$ has image
an irreducible curve $C'$ of degree $d$ with ${d-1}\choose{2}$ nodes.
Let $X$ be the surface obtained by blowing up $\PP^2$ at the nodes.
When $d\geq 6$ is even, the proper transform $C$ of $C'$ on $X$
has $C^2<0$, hence has primitivity 2. I.e., for this example we have
a primitive class $F$ of semi-effective order 2, and $2F$ is linearly equivalent to
a prime divisor of negative self-intersection.
\end{example}

We now give examples of primitive classes $F$ of semi-effective order 3 such that $3F$ is numerically equivalent to
a reduced (but not irreducible) curve $D$ with negative definite intersection matrix.

\begin{example}\label{example:LowPrimitivity3}\rm
Assume $K=\C$. Consider the $n^2+3$ points of intersection of the lines defined by the linear factors of
$H=(x^n-y^n)(x^n-z^n)(y^n-z^n)$ for $n\geq 3$, where $K[\PP^2]=K[x,y,z]$.
Let $X$ be the surface obtained by blowing up these $n^2+3$ points.
Let $D'$ be the curve defined by $H=0$, so $D'$ is the union of $3n$ lines, and $D'$
has multiplicity $n$ at 3 of the $n^2+3$ singular points of $D'$
(these 3 being the coordinate vertices of $\PP^2$) and multiplicity 3 at the remaining
singular points of $D'$.
Now let $D$ be the proper transform of $D'$ to $X$. Then, when 3 divides $n$, it is easy to see that
$D$ is $3F$ for a primitive class $F$ with semi-effective order 3, and $D$ has negative definite intersection matrix
(since the components of $D$ have negative self-intersection but are disjoint).
\end{example}

Example \ref{example:LowPrimitivity3} comes from a line arrangement
with no point where exactly two lines cross. Such examples are rare in
characteristic 0 (see \cite{BNAL} for the known examples, and see
\cite{BHRS} for curve arrangements using plane cubics which can be used in a similar way).
Examples are plentiful in positive characteristics,
and these lead to negative divisors where $kF$ is reduced and $k$ is arbitrarily large,
as we now show.

\begin{example}\label{example:LargePrimitivityPosChar}\rm
Let $K$ be an algebraically closed field of characteristic $p>0$.
Let $Q\subset K$ be a finite subfield, and let $q=|Q|$. There are $q^2+q+1$
$Q$-points in $\PP^2_K$ (i.e., points all of whose coordinates are in $Q$).
Let $P$ be one of them. Let $L_1,\ldots,L_r$
be all lines through pairs of $Q$-points which do not contain $P$
(there are $r=q^2$ such lines and $q^2+q$ such points). Let $D'$ be the union of these lines,
so $D'$ has degree $q^2$ and multiplicity $q$ at every $Q$-point other than $P$.
Blow up the the $Q$-points other than $P$ to obtain the surface $X$
and let $D$ be the proper transform of $C'$. Then
$D$ is $qF$ for a primitive class $F$ with semi-effective order $q$, and $D$ has negative definite self-intersection matrix
(since the components of $D$ have negative self-intersection but are disjoint).
\end{example}

We can construct examples with arbitrarily large semi-effective orders also in characteristic 0,
but the curves $D$ we obtain by our construction are not reduced or irreducible.

\begin{example}\label{example:LargePrimitivityChar0}\rm
Let $C'$ be a smooth plane curve of degree $d>1$. Pick distinct points $p_1,\ldots,p_r$ with $r\geq{d+1\choose2}$
such that there is no curve of degree less than $d$ containing the $r$ points. Let $p_{i,1}=p_i$ for each $i$,
let $p_{i,2}$ be the point of $C'$ infinitely near to $p_{i1}$, and iteratively for each $j>1$ let $p_{i,j+1}$ be the point of $C'$
infinitely near to $p_{i,j}$. Blow up the points $p_{i,1}$, then the points $p_{i,2}$, then $p_{i,3}$, etc.,
through $p_{i,j}$ for all $j\leq d$. Let the resulting surface be $X$. Let $E_{i,j}$ be the total transform to $X$ of the blow up
of $p_{i,j}$ and let $L$ be the total transform to $X$ of a general line in $\PP^2$. Then the linear equivalence divisor classes of
$L$ and $E_{i,j}$, $1\leq i\leq r$, $1\leq j\leq d$
give an integral orthogonal basis for the divisor class group of $X$, with respect to which
the proper transform $C$ of $C'$ is linearly equivalent to $dL-\sum_{i,j}E_{i,j}$.
Note also that there are prime divisors $N_{i,j}\sim E_{i,j}-E_{i,j+1}$ for all $1\leq i\leq r$ and $1\leq j < d$.
Let $D=C+\sum_i (N_{i,1}+2N_{i,2}+\cdots+(d-1)N_{i,d-1})\sim d(L-\sum_i E_{i,d})$.
Note that $F=L-\sum_i E_{i,d}$ is primitive and $dF$ is linearly equivalent to the effective divisor $D$.
If $bF$ were linearly equivalent to an effective divisor $B$ for some $b<d$, then the image $B'$ of $B$
in $\PP^2$ would be a curve of degree $b$ containing the points $p_1,\ldots, p_r$, contrary to our choice of the points $p_i$.
Thus $F$ has semi-effective order $d$. Finally, note that $C\sim dL-\sum_i (E_{i,1}+\cdots+E_{i,d})$, so
$N_{i,j}\cdot C=0$ for all $i$ and $j$.

Let $M=(A_i\cdot A_j)$ be the intersection matrix for $D$, so $A_i$ runs over the $r(d-1)+1$ components of $D$.
We have $C^2=d^2-rd\leq d^2-d^2(d+1)/2<0$.
Since $C\cdot N_{i,j}=0$ for all $i,j$ and $N_{i',j}\cdot N_{i,j}=0$ for all $j$ and all $i'\neq i$, $M$
can be taken to be a block diagonal matrix where one block is $M_0=(C^2)$, and there are $r$ blocks
consisting of submatrices $M_i=(N_{i,j}\cdot N_{i,j'})$, $1\leq i\leq r$.
The block $M_i$ is a $(d-1)\times (d-1)$ matrix with each diagonal entry being $-2$
and each entry to the immediate right or left of a diagonal entry being a 1, so
$$M_i=\left(\begin{array} {rrrrrrrrr}
-2 & 1 & 0 & 0 & \cdots & 0 & 0 & 0 & 0 \\
1 & -2 & 1 & 0 & \cdots & 0 & 0 & 0 & 0 \\
\vdots &  \vdots &  \vdots &  \vdots &   &  \vdots & \vdots & \vdots & \vdots \\
0 & 0 & 0 & 0 & \cdots & 0 & 1 & -2 & 1 \\
0 & 0 & 0 & 0 & \cdots & 0 & 0 & 1 & -2 \\
\end{array}
\right)
$$
Since the span $V_i$ of $N_{i,1},\ldots, N_{i,d-1}$ is a negative definite subspace of the divisor class group,
the matrix $M_i$ is also negative definite. (To see that $V_i$ is negative definite, note that
it has a rational orthogonal basis given by $E_{i,1}-E_{i,2}, E_{i,1}+E_{i,2}-2E_{i,3}, E_{i,1}+E_{i,2}+E_{i,3}-3E_{i,4},\ldots, E_{i,1}+\cdots+E_{i,d-2}-(d-2)E_{i,d-1}$,
and each of these classes has negative self-intersection.)
\end{example}

\begin{proof}[Proof of Theorem~\ref{introthm:divisors}]
Examples \ref{example:LargePrimitivityPosChar} and \ref{example:LargePrimitivityChar0} verify
the claim of the theorem.
\end{proof}



\newcommand\address[1]{\par\bigskip{\let\\=, #1}}
\newcommand\email[1]{\par\nopagebreak\textit{E-Mail address: }\texttt{#1}}
\footnotesize

\address{
   Thomas Bauer,
   Fachbereich Mathematik und Informatik,
   Philipps-Universit\"at Marburg,
   Hans-Meer\-wein-Stra\ss e,
   D-35032 Marburg, Germany.
}
\email{tbauer@mathematik.uni-marburg.de}

\address{Brian Harbourne,
   Department of Mathematics,
   University of Nebraska,
   Lincoln, NE 68588-0130 USA
}
\email{bharbourne1@unl.edu}

\address{
   Alex K\"uronya,
   Institut f\"ur Mathematik,
   Goethe-Universit\"at Frankfurt,
   Robert-Mayer-Str. 6-10.
   D-60325 Frankfurt am Main, Germany.
}
\email{kuronya@math.uni-frankfurt.de}

\address{
   Matthias Nickel,
   Institut f\"ur Mathematik,
   Goethe-Universit\"at Frankfurt,
   Robert-Mayer-Str. 6-10.
   D-60325 Frankfurt am Main, Germany.
   }

\email{nickel@math.uni-frankfurt.de}


\end{document}